\documentclass [12pt,reqno]{amsart}
\usepackage{amsmath}
\usepackage{amsthm}
\usepackage{dsfont}
\newtheorem{thm}{Theorem} 
\usepackage{amsfonts, amssymb}
\usepackage[utf8]{inputenc} 
\usepackage[english]{babel}
\usepackage{mathtools}
\usepackage{ mathrsfs }
\usepackage{enumerate}
\usepackage{ amssymb }
\usepackage{graphicx}
\usepackage{cite}

\usepackage{color}
\usepackage[dvipsnames]{xcolor}

\newtheorem{theorem}{Theorem}

\newtheorem{Question}[thm]{Question}
\newtheorem*{theorem*}{Theorem}

\newtheorem{lemma}[thm]{Lemma}
\newtheorem{claim}[thm]{Claim}
\newtheorem{proposition}[thm]{Proposition}

\theoremstyle{definition}

\theoremstyle{definition}

\theoremstyle{definition}
\newtheorem{definition}[thm]{Definition}

\newcommand{\SRA}{\operatorname{SRA}}

\newcommand{\CAT}{\operatorname{CAT}}

\newcommand{\ATB}{\operatorname{ATB}}
\newcommand{\dist}{\operatorname{dist}}
\newcommand{\Lip}{\operatorname{Lip}}
\newcommand{\Spt}{\operatorname{Spt}}
\newcommand{\diam}{\operatorname{diam}}

\numberwithin{equation}{section}

\def\R{{\mathbb R}}

\def\Z{{\mathbb Z}}

\def\N{{\mathbb N}}

\def\wt{\widetilde}
\def\wb{\overline}

\def\GG{{\Gamma}}

\def\<{\langle}
\def\>{\rangle}
\def \ee{{\epsilon}}
\def \dd{{\delta}}

\def \aa {{\alpha}}

\def \ll {{\lambda}}
\def \tt {{\theta}}

\def \zz {{\zeta}}

\begin{document}

\title{Bi-Lipschitz embeddings of $\SRA$-free spaces into Euclidean spaces}

\begin{abstract}
$\SRA$-free spaces is a wide class of metric spaces including  
 finite dimensional Alexandrov spaces of non-negative curvature, 
 complete Berwald spaces of nonnegative flag curvature, 
 Cayley Graphs of virtually abelian groups and  doubling metric spaces of non-positive Busemann 
 curvature with extendable geodesics.   
This class also includes arbitrary big balls in 
 complete, locally compact $\CAT(k)$-spaces $(k \in \R)$ with locally extendable geodesics, 
 finite-dimensional Alexandrov spaces of curvature $\ge k$ with $k \in R$
 and complete Finsler manifolds satisfying the doubling condition.  
 
We show that $\SRA$-free spaces allow bi-Lipschitz embeddings in Euclidean spaces.
As a corollary we obtain a quantitative bi-Lipschitz embedding theorem for balls in 
finite dimensional Alexandrov spaces of curvature bounded from below 
conjectured by S. Eriksson-Bique.

The main tool of the proof is an extension theorem for bi-Lipschitz maps into Euclidean spaces.
This extension theorem is close in nature with the embedding theorem of J. Seo and 
may be of independent interest.
\end{abstract}




 \author{Vladimir Zolotov}
 \address[Vladimir Zolotov]{Steklov Institute of Mathematics, Russian Academy of Sciences,
                     27 Fontanka,  St.Petersburg,  191023, Russia and  
					  Chebyshev Laboratory, St. Petersburg State University, 14th Line V.O., 29B, Saint Petersburg, 199178, Russia}
 \email[Vladimir Zolotov]{paranuel@mail.ru}

\maketitle

\section{Introduction}

\footnote{Acknowledgements: I thank my advisor Sergey V. Ivanov for all his ideas,
 advice and continuous support. The mission of investigating connections between $\SRA(\aa)$-condition
and bi-Lipschitz embeddability into Euclidean spaces was given to me by Alexander Lytchak. 
I'm very grateful for that.   I express my gratitude to Nina Lebedeva for multiple useful discussions. 
Research is supported by "Native towns", a social investment program of PJSC "Gazprom Neft".}

\begin{definition}\label{def:sra-free}
Let $k \in \N$ and $0 < \aa < 1$. We say that a metric space $X$ is free of $k$-point $\SRA(\aa)$-subspaces
if for every $k$-point subset $Y \subset X$ there exist
$x,z,y \in Y$ such that 
$$d(x,y) > \max\{d(x,z) + \aa d(z,y), \aa d(x,z) +  d(z,y)\}.$$
We say that a metric space is an $\SRA(\aa)$-free space if it is free of $k$-point $\SRA(\aa)$-subspaces
for some $k \in \N$.
\end{definition}


The condition of being $\SRA(\aa)$-free is a strengthening of the doubling condition (see \cite[Theorem 6]{LOZ}).
The reason why $\SRA(\aa)$-free spaces were studied in previous works \cite{LOZ, ZSC}  is the following.
By showing that a space is an $\SRA(\aa)$-free space we are showing that:
\begin{itemize}
\item{this space does not contain large $\aa$-snowflakes as isometric subspaces (see \cite[Proof of Theorem 1.1]{BS}). 
}
\item{bounded self-contracted curves are rectifiable in this space
(see \cite[Theorem 1]{ZSC}).}
\end{itemize}

The aim of this paper is to give a new application for the concept by proving that 
 every  $\SRA(\aa)$-free space allows a bi-Lipschitz  embedding into some Euclidean space.
 
 In the present paper $\R^n$ is always considered as metric spaces with Euclidean metric.
 
\begin{theorem}\label{EmbThm}
For $k \in \N$ and $0 < \aa < 1$ there exist $N \in \N$ and $D \ge 1$ satisfying the following.
For every metric space $X$ which is free of $k$-point $\SRA(\aa)$-subspaces there exists 
a bi-Lipschitz embedding $\Phi:X \rightarrow \R^N$ which bi-Lipschitz
distortion does not exceed $D$.  
\end{theorem}


The class of $\SRA(\aa)$-free spaces is huge, it includes 
\begin{itemize}
\item{finite dimensional normed spaces 
(see \cite[Proposition 15]{LOZ} and \cite[Theorem 2]{LOZ}),}
\item{finite dimensional Alexandrov spaces of non-negative curvature
 (see \cite[Proposition 14]{LOZ} and \cite[Theorem 2]{LOZ}),}
\item{complete Berwald spaces of nonnegative flag curvature 
(see \cite[ Proposition 17]{LOZ}) and \cite[Theorem 2]{LOZ}),}
\item{Cayley Graphs of virtually abelian groups 
(see \cite[Proposition 18]{LOZ}  and \cite[Theorem 2]{LOZ}),}
\item{doubling metric spaces of non-positive Busemann curvature with extendable geodesics 
(see \cite[Proposition 12]{LOZ}  and \cite[Theorem 2]{LOZ}),}
\item{globally $\ATB(\ee)$-spaces 
(see Definition \ref{ATB} and \cite[Theorem 2]{LOZ})).}
\end{itemize}
It also includes arbitrary big balls in following classes of spaces:
\begin{itemize}
\item{complete, locally compact $\CAT(k)$-spaces $(k \in \R)$ with locally extendable geodesics 
(see \cite[Proposition 12]{LOZ}, \cite[Theorem 2]{LOZ}) and \cite[Theorem 5]{ZSC},}
\item{finite-dimensional Alexandrov spaces of curvature $\ge k$ with $k \in R$ 
(see \cite[Proposition 14]{LOZ}, \cite[Theorem 2]{LOZ}) and \cite[Theorem 6]{ZSC},}
\item{complete locally compact Finsler manifolds  
(see \cite[Proposition 16]{LOZ}, \cite[Theorem 2]{LOZ}) and \cite[Theorem 5]{ZSC},}
\item{proper locally $\ATB(\ee)$-spaces 
(see Definition \ref{ATB}, \cite[Theorem 2]{LOZ}) and \cite[Theorem 5]{ZSC}).}
\end{itemize}

Thus, Theorem \ref{EmbThm} generalizes several previously known results
on bi-Lipschitz embeddings into Euclidean spaces, see \cite[Theorem 1–3, Theorem 1–4]{EB},
 \cite[Therorem 4.5]{LP-BilEmb},
and its generalization \cite[Theorem 4]{ZSC}. On the other hand 
Theorem \ref{EmbThm} does not provide a criterion for embeddability into Euclidean 
spaces i.e., there are metric spaces which allow bi-Lipschitz embeddings into 
Euclidean spaces which are not  $\SRA(\aa)$-free spaces for any $0 < \aa < 1$. 

As an application of Theorem \ref{EmbThm} we prove the following  
embedding result for Alexandrov spaces of curvature bounded from below, 
conjectured by S. Eriksson-Bique \cite[Conjecture 1–9]{EB}.
 
\begin{theorem} \label{AlexThm}
For $n \in \N$, $k < 0$ and $R > 0$ there exist $D > 0$ and $N \in \N$ satisfying the following.
For every $n$-dimensional Alexandrov space of curvature $\ge k$ and every $x \in X$
there exists an embedding $\phi:B_R(x) \rightarrow \R^N$ which 
bi-Lipschitz distortion does not exceed $D$.

For $n \in N$ there exist $D_0 > 0$ and $N_0 \in \N$ such that 
every $n$-dimensional Alexandrov space of non-negative curvature 
allows an embedding into $\R^{N_0}$ which 
bi-Lipschitz distortion does not exceed $D_0$.
\end{theorem}

The main tool for the proof of Theorem \ref{EmbThm} is the following extension theorem
for bi-Lipschitz maps into Euclidean spaces.

\begin{theorem}\label{ExtThm}
Let $\ll \in \N$ and $X$ be a metric space with the doubling constant $\ll$.
Let $Y \subset X$ and $\phi: Y \rightarrow \R^n$ be a bi-Lipschitz embedding into some Euclidean space.
Suppose that there exist $D, \wb n  > 0$ and $0 < \tt \le 1$ such that for every $x \in X$ there exists an embedding 
of $B_{\tt d(x,Y)}(x)$ into $\R^{\wb n}$ which bi-Lipschitz distortion does not exceed $D$.
Then there exist $\wt \phi:X \rightarrow \R^{n} \times \R^{N} = \R^{n + N}$ such that 
\begin{enumerate}
\item{$N$ and $\dist(\wb \phi)$ are bounded by some functions of $n$, $\wb n$, $\dist(\phi)$,  $D$, $\tt$ and $\ll$.}\label{Bound}
\item{for every $y \in Y$ we have $\wb \phi(y) = (\phi(y) , 0).$} \label{Ext}
\end{enumerate}
\end{theorem}

The first remark about Theorem \ref{ExtThm} we should do  is that the second part 
of the conclusion (\ref{Ext}) is not that important. Let us denote by Theorem \ref{ExtThm}-emb
a version of Theorem \ref{ExtThm} where (\ref{Ext}) is skipped. Theorem \ref{ExtThm}-emb
implies Theorem \ref{ExtThm} if combined with the following proposition.
\begin{proposition}(\cite[Theorem 1.2]{MMMR}, see also \cite[Theorem
5.5]{AV97})
Let $X \subset \R^n$ and $f:X \rightarrow \R^m$ be a bi-Lipschitz map with distortion 
at most $D$. There exists  $f': \R^n \rightarrow \R^{m+n}$ with the distortion at most  $3D$
and such that  
$$f'(x) = (f(x),0),$$
for every $x \in X$.
\end{proposition}
  
 Our second remark is that Theorem \ref{ExtThm} seems to address the same phenomena as 
the embedding result by J. Seo \cite[Theorem  1.1]{Seo}.
It also seems very plausible that by applying \cite[Lemma 2.5]{Romney} and 
some standard arguments one can show that 
Theorem \ref{ExtThm}-emb is equivalent to Romney's version of Seo's result \cite[Theorem 2.3]{Romney}.
{\color{Black}
We have not verify carefully the validity of this approach and give a direct proof of Theorem \ref{ExtThm} instead.
}

\section{Proof of Theorem \ref{ExtThm}}

For a map between metric spaces $\phi:X \rightarrow Y$ we denote by $\Lip(\phi)$ the Lipschitz constant of $\phi$ 
{\color{Black} i.e.,
$$\Lip(\phi) = \sup_{x_1 \ne x_2 \in X} \frac{d_Y(\phi(x_1), \phi(x_2))}{d_X(x_1,x_2)} \in [0, \infty] ,$$
and by $\dist(\phi)$ its (bi-Lipschitz) distortion, which the infimum over those $D \in [1, \infty]$ such that 
there exists $s > 0$ satisfying 
$$sd_X(x_1,x_2) \le d_Y(\phi(x_1),\phi(x_2)) \le Dsd_X(x_1,x_2)$$
for all $x_1,x_2 \in X$.}

A metric space is said to have the doubling constant $\ll \in \N$ if for every $x \in X$ and $R > 0$ 
there exists $x_1,\dots,x_\ll \in X$ such that $B_R(x) \subset \cup_{i=1}^\ll B_{R/2}(x_i)$.

 We denote the standard Euclidean norm by $\vert \vert \cdot \vert \vert $. 

 {\color{Black}It is well known that partial Lipschitz maps to Euclidean spaces allow extensions to whole spaces.}
\begin{proposition}[McShane]\label{McShane}
Let $X$ be a metric space, $Y \subset X$, $L > 0$, $f:Y \rightarrow \R^n$ be a $L$-Lipschitz
map then there exists $F:X \rightarrow  \R^n$ such that $F$ is $\sqrt{n}L$-Lipschitz and 
$F|_{Y} = f$.
\end{proposition}

Suppose that we start  from a bi-Lipschitz function $f$ and apply the previous 
proposition. The resulting map is not necessarily bi-Lipschitz.
 {\color{Black} But if we have 
two points $x_1, x_2 \in X$ such that both of them  are 
close to $Y$ and they are far from each other then $F$ will not distort distance between 
them too much.} The following lemma formalizes this statement.



{\color{Black}
\begin{lemma}\label{ExtLem1v2}
Let $X, Z$ be metric spaces, $Y \subset X$, $\phi:Y \rightarrow Z$ be a map having bi-Lipschitz distortion 
$dist(\phi) = D \in [1, \infty)$ for some $s > 0$ i.e., 
\begin{equation} sd(y_1,y_2) \le \vert \vert  \phi(y_1) - \phi(y_2) \vert \vert  \le Dsd(y_1, y_2),\label{biLip}\end{equation} 
for every $y_1, y_2 \in Y$. Let $\nu \ge 1$ and $\phi_1: X \rightarrow Z$ be a 
$Ds\nu$-Lipschitz extension of $\phi$. Then 
\begin{enumerate}
\item{for every $K \ge 0$ if $x_1, x_2 \in X$ satisfy 
$$d(x_1,x_2) > K\nu\max\{d(x_1,Y), d(x_2,Y)\}$$ then 
$\vert \vert  \phi_1(x_1) - \phi_1(x_2)\vert \vert  \ge s(1 - \frac{2}{K\nu} - \frac{2D}{K})d(x_1, x_2), \label{CoLipGen}$}
\item{in particular if $\nu = \sqrt{n}$ and $x_1, x_2 \in X$ satisfy 
$$d(x_1,x_2) > 5D\sqrt{n}\max\{d(x_1,Y), d(x_2,Y)\}$$ then 
$\vert \vert  \phi_1(x_1) - \phi_1(x_2)\vert \vert  \ge \frac{s}{5}d(x_1, x_2). \label{CoLip}$}
\end{enumerate}
\end{lemma}
}
\begin{proof}[Proof of (\ref{CoLipGen})]

Indeed, fix $x_1,x_2 \in X$ such that $d(x_1,x_2) > K\nu\max\{d(x_1,Y), d(x_2,Y)\}$.
Fix $\dd$ such that $\max\{d(x_1,Y), d(x_2,Y)\} < \dd < \frac{d(x_1,x_2)}{K\nu}$
 and $y_1, y_2 \in Y$ be such that $d(x_1, y_1), d(x_2, y_2) < \dd$.
 By the triangle inequality we have
 \begin{equation} \vert \vert  \phi_1(x_1) - \phi_1(x_2) \vert \vert  \ge \vert \vert  \phi_1(y_1) - \phi_1(y_2) \vert \vert  - \vert \vert   \phi_1(y_1) - \phi_1(x_1) \vert \vert  
 - \vert \vert   \phi_1(y_2) - \phi_1(x_2) \vert \vert  . \label{TIneq}\end{equation}
By the left part of (\ref{biLip}) we have 
$$\vert \vert  \phi_1(y_1) - \phi_1(y_2) \vert \vert  = \vert \vert  \phi(y_1) - \phi(y_2) \vert \vert   \ge s d(y_1, y_2).$$
And since $\phi_1$ is $Ds\nu$-Lipschitz we have 
$$\vert \vert   \phi_1(y_1) - \phi_1(x_1) \vert \vert   \le Ds\nu d(y_1, x_1) \le Ds\sqrt{n}\dd,$$
$$\vert \vert   \phi_1(y_2) - \phi_1(x_2) \vert \vert   \le Ds\nu d(y_2, x_2) \le Ds\sqrt{n}\dd.$$
Substituting last three inequalities into (\ref{TIneq}) provides 
$$ \vert \vert  \phi_1(x_1) - \phi_1(x_2) \vert \vert  \ge  s d(y_1, y_2) - 2Ds\nu\dd. \ge  $$
$$ \ge  s\big(d(x_1,x_2) - d(x_1, y_1) - d(x_2, y_2)\big) -  2Ds\nu\dd \ge$$
$$ \ge s\big(d(x_1,x_2) - 2\dd - 2D\nu\dd\big).$$
Since $ \dd < \frac{d(x_1,x_2)}{K\nu}$ we have 
$$ \vert \vert  \phi_1(x_1) - \phi_1(x_2) \vert \vert   \ge sd(x_1,x_2)(1 - \frac{2}{K\nu} - \frac{2D\nu}{K\nu}).$$
And we are done.
\end{proof}
\begin{proof}[Proof of (\ref{CoLip})] We substitute $\nu = \sqrt{n}$ and $K = 5D$ into the last inequality 
of the previous proof. This provides
$$\vert \vert  \phi_1(x_1) - \phi_1(x_2) \vert \vert   \ge sd(x_1,x_2)(1 - \frac{2}{K\nu} - \frac{2D}{K}) \ge \frac{s}{5}d(x_1,x_2).$$
\end{proof}


\begin{lemma}\label{ExtLem2}
Let $X$ be a metric space with the doubling constant $\ll \in \N$, $\GG \ge 1$, $f \ge 0$ be a $\GG$-Lipschitz function on X.
Suppose that there are $n \in \N$ and $D \ge 1$ such that for every $x \in X$ there exists an embedding of $B_{f(x)}(x)$ in $\R^n$
with distortion less then $D$. Then for every $\zz \ge 1$ there exist
$N = N(n, D, \ll, \zz, \GG)$ and
 $\Phi:X \rightarrow \R^N$ such that 
\begin{enumerate}
\item{if $f(x) = 0$ then $\Phi(x) = 0$,\label{CZero}}
\item{$\Lip(\Phi)$ is bounded by some function of $n$, $D$, $\ll$, $\zz$ and $\GG$,\label{CDepends}}
\item{if $x_1, x_2 \in X$ satisfy $d(x_1,x_2) \le \zz \max\{f(x_1), f(x_2)\}$ then 
$\vert \vert  \Phi(x_1) - \Phi(x_2)\vert \vert  \ge  Kd(x_1,x_2)$, for $K =   \frac{9\GG}{40\zz}$. \label{CoLip2}}
\end{enumerate}
\end{lemma}
\begin{proof}
We will only deal with the case $\GG = 1$ and the general case {\color{Black}follows by  rescalings} of {\color{Black} metrics}. 
Most of ideas of the following proof came from the proof of Assouad's embedding theorem (see \cite{Assouad, Hei, NN10}). 

\textit{Step 1: maps $P_x$.} On this step we introduce maps $P_x$ which will be used 
as building blocks for $\Phi$.

A subset of a metric space is said to be \emph{$r$-separated}
if every pair of distinct points in that set is of distance $\ge r$.
For every $k \in \Z$ we fix $N_k$ to be a maximal $\frac{2^k}{10}$-separated subset in $\{x \in X\vert 2^k \le f(x) \le 2^{k+2}\}$.
In the later text we assume that for $k \neq l \in \Z$ we have $N_k \cap N_l = \emptyset$. The proof of the general case is 
the same but requires more messy notation.

For $x \in N_k$ we construct $P_x:X \rightarrow \R^{n+1}$ such that 
\begin{enumerate}
\item{${P_x}(B_{2^{k-1}}(x)) \subset \{(y_1,\dots,y_{n+1}) \in \R^{n+1} \vert \vert y_1 + \dots + y_{n+1} = 2^k \sqrt{n+1}\}$, \label{PxInCS}}
\item{$P_x(x) = \frac{2^k}{\sqrt{n+1}}(1,\dots,1),$}
\item{${P_x}|_{B_{2^{k-1}}(x)}$ has bi-Lipshitz distortion $D$ for $s = 1$, \label{PxBiLip}}
\item{$P_x(X \setminus B_{\frac{11}{10}2^{k-1}}(x)) = 0$. \label{PxZero}}
\item{$P_x$ is $40D\sqrt{n+1}$-Lipschitz. \label{PxLip}}
\end{enumerate}
Construction. First we define $P_x$ only on $B_{2^{k-1}}(x)$ and in a way 
that it satisfies (1-3) which can be done by assumptions of the lemma.  Then we define $P_x$ on $X \setminus B_{\frac{11}{10}2^{k-1}}(x)$
by zero. For every $z \in B_{2^{k-1}}(x)$ we have 
$$\vert \vert  P_x(z) \vert \vert  \le  
\vert \vert  P_x(x) \vert \vert  + \vert \vert  P_x(z) - P_x(x)\vert \vert  \le 2^k + (D2^{k-1}) \le 2^{k+1}D.$$
This implies that $P_x$ is $40D$-Lipschitz on $B_{2^{k-1}}(x) \cup (X \setminus B_{\frac{11}{10}2^{k-1}}(x))$. 
By  Proposition \ref{McShane} $P_x$  can be extended  to the whole $X$ in a way that the resulting map is  $40D\sqrt{n + 1}$-Lipschitz.

\textit{Step 2: coloring  $\chi$.} On this step we introduce coloring $\chi$ of $\cup_{k = -\infty}^{\infty}N_k$.
It will assist us in our mission of constructing a finite dimensional map $\Phi$ from the countable 
set of maps $\{P_x\}_{x \in \cup_{k = -\infty}^{\infty}N_k}$.
 
For $a, R > 0$ and a metric space having doubling constant $\ll$ we have that the size of 
an $a$-separated subset of a ball of  radius $R$ does not exceed $\ll^{2+ \log_2(\frac{R}{a})}$. 
Thus if we take $j = \ll^{2 + \log_2(100\zz)}$  then for every $x \in X$  the size of 
 $\{y \in N_k\vert  d(x,y) < 10\cdot2^k \zz\}$ is less or equal to $j$. We fix a coloring 
 $\chi:\cup_{k = -\infty}^{\infty}N_k \rightarrow \{1,\dots,j\}$ such that for every $k \in \Z$ and every $x,y \in N_k$ equality
 $\chi(x) = \chi(y)$ implies
 \begin{equation} d(x,y) \ge 10\cdot2^k \zz. \label{MAFar}\end{equation}
 
 \textit{Step 3: Construction of $\Phi$.} We will start from the set of maps 
 $\{P_x\}_{x \in \cup_{k = -\infty}^{\infty}N_k}$. We are going to split them into finite amount of groups  in such a manner 
 that supports of elements of the same group are far from each other (see Claim \ref{TheClaim}).
 From each group we will get a map $Q_{\wt k, \wt j }$ by simply taking the sum of all elements.
 Finally, we will bundle all those maps $Q_{\wt k, \wt j }$ together to get $\Phi$. 
 More precisely we do the following.  
  
Let $M = \max\{8, \log_2(440D\sqrt{n+1})\}$. For in $a \in Z$ we denote 
$$\wb a = \{b \in \Z \vert \vert b = a\text{ mod }M\}.$$  
For $\wt k = \wb 1,\dots,\wb {M}$ and $1 \le \wt j  \le j$ we define $Q_{\wt k, \wt j }:X \rightarrow \R^{n+1}$ by 
$$Q_{\wt k, \wt j }(z) = \sum_{k \in \wt k}\sum_{x \in N_k, \chi(x) = \wt j}P_x(z).$$
We claim that a map $\Phi:X \rightarrow \R^N = \R^{(n+1)Mj}$ defined by 
$$\Phi(z) =  \Big(Q_{\wt k, \wt j }(z)\Big)_{\wt k = \wb 1,\dots,\wb M, 1 \le \wt j \le j}.$$
satisfies conditions (\ref{CZero} - \ref{CoLip2}).

 \textit{Step 3: proof that (\ref{CZero}) is satisfied.} 

 Fix $x \in N_k$, we have $f(x) \ge 2^k$ and $f$  is $1$-Lipschitz
thus $d(x,\{p \in X  \vert f(p) = 0\}) \ge 2^k$. Since $P_x(X \setminus B_{\frac{11}{10}2^{k-1}}(x)) = 0$ we conclude that 
$P_x(\{p \in X\vert  f(p) = 0\}) = 0$. Thus $\Phi(\{p \in X \vert f(p) = 0\})  = 0$ too.

 \textit{Step 4: Claim \ref{TheClaim}.} 
The following claim says that supports of 
$P_{x_1}$, $P_{x_2}$ which are summands in the same $Q_{\wt k, \wt j }$ are far from each other. 
This claim will we useful for proving both (\ref{CDepends}) and (\ref{CoLip2}).

\begin{claim}\label{TheClaim}
Let $k_1, k_2 \in \Z$ be such that 
$k_1 = k_2 \mod M$, $x_1 \in N_{k_1}$, $x_2 \in N_{k_2}$, $z_1 \in \Spt(P_{x_1})$, $z_2 \in \Spt(P_{x_2})$. 
If $\chi(x_1) = \chi(x_2)$ and $x_1 \neq x_2$ 
then
$$d(z_1,z_2) \ge \frac{2}{5}2^{\max\{k_1,k_2\}}.$$
\end{claim}
\begin{proof}
The case $k_1 = k_2$ follows from (\ref{MAFar}). In the case $k_1 \neq k_2$ we can assume $k_1 \ge k_2 + M$ without loss
of generality. We have that $f(x_1) \ge 2^{k_1}$. Since $f$ is $1$-Lipschitz and $\Spt(P_{x_1}) \in \wb{B_{\frac{11}{10}2^{k_1-1}}(x_1)}$
we conclude that $f(z_1) \ge \frac{9}{10}2^{k_1-1}$. The similar argument shows that $f(z_2) \le 2^{k_2+3} \le 2^{k_1 - (M - 3)}$.
Thus $f(z_1) - f(z_2) \ge \frac{2}{5}2^{k_1}$. Once again we apply that $f$ is $1$-Lipschitz and obtain 
$d(z_1, z_2) \ge \frac{2}{5}2^{k_1}$. 
\end{proof}
\textit{Step 5: proof of (\ref{CDepends}).} 
  We are going to show that 
$$Q_{\wt k, \wt j } = \sum_{k \in \wt k}\sum_{x \in N_k, \chi(x) = \wt j}P_x$$  
  is Lipschitz for every $\wt k = \wb 1,\dots,\wb {M}$ and $1 \le \wt j  \le j$.
By Claim \ref{TheClaim} we know that supports of functions $P_x$ in the sum do not intersect each other.
Fix $z_1, z_2 \in X$. If there exists $k \in \wt k$ and $x_1 \in N_k$ such that  $\chi(x_1) = \wt j$ and $z_1, z_2 \in \Spt(P_{x_1})$ then 
we have
 $$\vert \vert  Q_{\wt k, \wt j }(z_1) - Q_{\wt k, \wt j }(z_2) \vert \vert \vert
 = \vert \vert  P_{x_1}(z_1) - P_{x_1}(z_2) \vert \vert  \le 
40D\sqrt{n+1}d(z_1,z_2)$$
by (\ref{PxLip}).  The same is true if at least one  of  points $z_1, z_2 \in X$ is outside of supports of all functions $P_x$. 

The remaining case {\color{Black} is if} there exist $k_1, k_2 \in \wt k$ and different $x_1 \in N_{k_1}, x_2 \in N_{k_2}$ such that 
$\chi(x_1) = \chi(x_2) = \wt j$ and $z_1 \in \Spt(P_{x_1})$, $z_2 \in \Spt(P_{x_2})$. In this case by the claim we have 
\begin{equation} d(z_1,z_2) \ge \frac{2}{5}2^{\max\{k_1,k_2\}}. \label{PxFar} \end{equation}
On the other hand since each $P_{x_i}$ is  $40D\sqrt{n + 1}$-Lipschitz and $\Spt(P_{x_i})  \subset \wb{B_{\frac{11}{10}2^{k_i-1}}(x)}$ we have 
$$\vert \vert  P_{x_i}(z_i) \vert \vert  \le  22D\sqrt{n + 1}2^{k_i},$$
for $i = 1,2.$
Thus,
$$\vert \vert  Q_{\wt k, \wt j }(z_1) - Q_{\wt k, \wt j }(z_2) \vert \vert  = \vert \vert  P_{x_1}(z_1) - P_{x_2}(z_2) \vert \vert  \le 44D\sqrt{n + 1}2^{\max\{k_1,k_2\}}.$$
Combining this with (\ref{PxFar}) provides 
$$\vert \vert  Q_{\wt k, \wt j }(z_1) - Q_{\wt k, \wt j }(z_2) \vert \vert   \le  110D\sqrt{n + 1} d(z_1,z_2).$$
Thus we conclude that $Q_{\wt k, \wt j }$ is $110D\sqrt{n + 1}$-Lipschitz   for every $\wt k = \wb 1,\dots,\wb {M}$ and $1 \le \wt j  \le j$.
And $\Phi$ is  $110D\sqrt{(n + 1)Mj}$-Lipschitz.

\textit{Step 6: proof of (\ref{CoLip2}).} 
 Let $z_1, z_2 \in X$ be such that 
$d(z_1,z_2) \le \zz \max\{f(z_1), f(z_2)\}$.  We have to show that 
$$\vert \vert  \Phi(z_1) - \Phi(z_2)\vert \vert  \ge Kd(z_1,z_2).$$

Without loss of generality we assume that $f(z_1) \ge f(z_2)$. We fix $k \in \N$ such that $f(z_1) \in [2^k, 2^{k+1})$.
  
The first case is $d(z_1, z_2) < \frac{4}{5}2^{k - 1}$. By the construction of $N_k$ there exists $x \in N_k$ such that 
$d(z_1, x) < \frac{1}{5}2^{k-1}$. Then we also have $d(z_2, x) <  2^{k-1}$ by the triangle inequality. 
Thus by the property (\ref{PxBiLip}) from the construction of $P_x$ we have 
$$\vert \vert  P_x(z_1) - P_x(z_2)\vert \vert  \ge  d(z_1,z_2).$$
By the Claim \ref{TheClaim} we conclude that 
$$\vert \vert  Q_{\wb k, \chi(x)}(z_1) - Q_{\wb k, \chi(x)}(z_2)\vert \vert  \ge d(z_1,z_2).$$
And thus 
$$\vert \vert  \Phi(z_1) - \Phi(z_2)\vert \vert  \ge d(z_1,z_2)  \ge Kd(z_1,z_2).$$

Now lets consider the case $d(z_1, z_2) \ge \frac{4}{5}2^{k - 1}$.  By the construction of $N_{k-1}$ there exists $x \in N_{k-1}$ such that 
$d(z_1, x) < \frac{1}{5}2^{k-2}$. By the triangle inequality we conclude that
 $d(z_2,x) \ge \frac{7}{5}2^{k-2}$.  
 By the property (\ref{PxInCS}) from the construction of $P_x$ we have 
 $$\vert \vert  P_x(z_1) \vert \vert  \ge 2^{k-1}.$$
And thus 
$$\vert \vert  Q_{\wb{k - 1}, \chi(x)}(z_1) \vert \vert  \ge 2^{k-1}.$$
Now we are going to show that
 $$\vert \vert  Q_{\wb{k - 1}, \chi(x)}(z_2) \vert \vert  = 
\vert \vert  \sum_{k' - 1  \in \wb{k - 1}}\sum_{x' \in N_{k' - 1}, \chi(x')  = \chi(x)}P_{x'}(z_2) \vert \vert  \le \frac{1}{10} 2^{k-1}.$$
By the Claim \ref{TheClaim} we know that at most one $P_{x'}(z_2)$ in the previous sum is non zero.
The case when all summands are zero is trivial. Thus we assume that non-zero summand exist 
and denote by $k'$ and $x'$ corresponding indexes. First lets note that $k' \le k$. Indeed, suppose that $k' \ge k+M$.
Then $f(x') \ge 2^{k' - 1}$ and since $f$ is $1$-Lipschitz and
  $\Spt(P_{x'})  \subset \wb{B_{\frac{11}{10}2^{k'-2}}(x')}$
  we conclude that $f(z_2) \ge \frac{9}{10}2^{k'-2} > 2^{k+1} \ge f(z_1)$. Which contracts 
  the assumption that $f(z_2) \le f(z_1)$. 
 
 Next we are going to show that $k' \neq k$. Once again we argue by contradiction and suppose that 
 $k' = k$. Combining $d(z_2,x) \ge \frac{7}{5}2^{k-2}$ and the property (\ref{PxZero}) of $P_x$ we obtain that 
$$P_x(z_2) = 0.$$ Thus $x' \neq x$.
 Then by the construction of $N_{k-1}$ we have that
 $$d(x,x') \ge 10\cdot 2^{k-1}\zz.$$
Since $z_1 \in \Spt(P_{x})  \subset \wb{B_{\frac{11}{10}2^{k-2}}(x)}$ and 
$z_2 \in \Spt(P_{x'})  \subset \wb{B_{\frac{11}{10}2^{k-2}}(x')}$ the previous inequality implies 
$$d(z_1,z_2) \ge 8\cdot 2^{k-1}\zz.$$
which contradicts the assumption that 
$$d(z_1,z_2) \le \zz \max\{f(z_1), f(z_2)\} \le 2^{k+1} \zz.$$
We conclude that the only possible case is that $k' \le k - M$. Thus we have 
$$\vert \vert  P_{x'}(z_2) \vert \vert  \le \vert \vert  P_{x'}(x') \vert \vert  + \Lip(P_{x'})\frac{11}{10}2^{k' - 2} \le$$
$$ \le 2^{k - 1 - M}  + 40D\sqrt{n+1}\frac{11}{10}2^{k - 2 - M} \le$$
$$\le 2^{k - 1 - M} + 44D\sqrt{n+1}2^{k - 2 - M} = 2^{k-1}2^{-M}44D\sqrt{n+1} \le 2^{k-1}\frac{1}{10},$$
where the last inequality follows from the definition of $M$.
Thus we have that  
$$\vert \vert  Q_{\wb{k - 1}, \chi(x)}(z_2) \vert \vert  \le \frac{1}{10}2^{k-1}.$$
Which implies 
$$\vert \vert  \Phi(z_1) - \Phi(z_2)\vert \vert  \ge \vert \vert  Q_{\wb{k - 1}, \chi(x)}(z_1) -  Q_{\wb{k - 1}, \chi(x)}(z_2)\vert \vert  \ge \frac{9}{10}2^{k-1}.$$
In combination with 
$$d(z_1,z_2) \le \zz \max\{f(z_1), f(z_2)\} \le \zz  2^{k+1}$$
the previous inequality implies
$$\vert \vert  \Phi(z_1) - \Phi(z_2)\vert \vert  \ge  \frac{9}{40\zz} d(z_1,z_2) = Kd(z_1,z_2).$$
\end{proof}

\begin{proof}[Proof of Theorem \ref{ExtThm}]
Suppose that we are in conditions of Theorem \ref{ExtThm}. Without loss of generality 
we can assume that 
$$d(y_1 ,y_2) \le \vert \vert \phi(y_1) - \phi(y_1) \vert \vert \le \dist(\phi) d(y_1 ,y_2).$$
Let $\phi_1$ be the extension of $\phi$ provided by Proposition \ref{McShane}.  We define 
$f:X \rightarrow [0, \infty)$ by 
$$f(x) = \tt \dist(x,Y).$$
And we take $\zz = \frac{5\dist(\phi)\sqrt{n}}{\tt}$.  Let $\Phi:X \rightarrow \R^N$  be the 
map provided by Lemma \ref{ExtLem2} for those $f, \zz$ and $\GG = 1$. 
We claim that the map $\wt \phi = (\phi_1, \Phi)$ satisfies the conclusion of the Theorem \ref{ExtThm}. 

Indeed, (\ref{Ext}) from the conclusion of Theorem \ref{ExtThm}
 follows from (\ref{CZero}) from the  conclusion of Lemma \ref{ExtLem2}. 
Once again $N$ comes from Lemma \ref{ExtLem2} with the required upper bound
which depends only on  $n$, $\wb n$, $\dist(\phi)$,  $D$, $\tt$ and $\ll$.
Note that both $\phi_1$ are $\Phi$ are Lipschitz and their Lipschitz 
constants depend only  $n$, $\wb n$, $\dist(\phi)$,  $D$, $\tt$ and $\ll$.
Thus, the same is true for $\wt \phi$.
 
Finally we have to provide the inequality
$$\vert \vert  \wt \phi(x_1) - \wt \phi(x_2) \vert \vert  \ge C(n, \wb n, \dist(\phi),  D, \tt, \ll)d(x_1,x_2),$$  
for every $x_1, x_2 \in X$.
 
In the case 
$$d(x_1,x_2) > 5\dist(\phi)\sqrt{n} \max\{d(x_1 ,Y ), d(x_2 ,Y )\},$$
we have 
$$\vert \vert  \wt \phi(x_1) - \wt \phi(x_2) \vert \vert \ge 
\vert \vert   \phi_1(x_1) -  \phi_1(x_2) \vert \vert \ge \frac{1}{5} d(x_1, x_2),$$
where the last inequality follows from Lemma \ref{ExtLem2}(\ref{CoLip}). 

In the case
$$d(x_1,x_2) \le 5\dist(\phi)\sqrt{n} \max\{d(x_1 ,Y ),d(x_2 ,Y )\} = \zz \max\{f(x_1),f(x_2)\},$$
we have 
$$\vert \vert  \wt \phi(x_1) - \wt \phi(x_2) \vert \vert \ge 
\vert \vert   \Phi(x_1) -  \Phi(x_2) \vert \vert \ge 
\frac{9}{40 \zz}d(x_1, x_2) = 
\frac{9\tt}{200\dist(\phi)\sqrt{n}}d(x_1, x_2),$$
where the last inequality comes from Lemma \ref{ExtLem2}(\ref{CoLip2}).
\end{proof}

\section{Proof of Theorem \ref{EmbThm}}

\begin{definition}[Small rough angle condition; $\SRA(\aa)$]\label{def:sra}
Let $(X,d)$ be a metric space and $0 < \aa < 1$.
We say that $X$ satisfies the \emph{$\SRA(\aa)$-condition} if,
for every $x,y,z \in X$, we have 
\begin{equation}
d(x,y) \le \max\{d(x,z) + \aa d(z,y), \aa d(x,z) + d(z,y)\}. \label{SRA-ineq}
\end{equation}
\end{definition}

\begin{lemma}\label{LocEmb}
Let  $0 < \aa < 1$ and $X$ be a metric space such that it is free of $k$-point $\SRA(\aa)$-subspaces.
Suppose that $\{x_1,\dots,x_{k-1}\} \subset X$ taken 
with induced metrics satisfies $\SRA(\frac{\aa}{2})$-condition.
Then $B_{\frac{\aa R}{6}}(x_1) \subset X$ could be embedded into $\R^{k-2}$ with distortion $ \frac{\sqrt{k-2}}{\aa}$,
where $R = \min_{1 \le i < j \le k - 1}{d(x_i,x_j)}$.
\end{lemma}
\begin{proof}
We claim that a map $\phi:X \rightarrow \R^{k-2}$ given by 
$$\phi_i(y) = \dist(y, x_{i-1})$$
does the trick. Clearly $\phi$ is a $\sqrt{k-2}$-Lipschitz map. It suffices to show that 
for every $y_1, y_2 \in B_{\frac{\aa R}{6}}(x_1)$ we have 
$$\vert \vert  \phi(y_1) - \phi(y_2) \vert \vert  \ge \aa d(y_1, y_2).$$
Fix $y_1, y_2$. We claim that there exists $1 \le m \le k-2$ such that 
$$\vert   \phi_m(y_1) - \phi_m(y_2) \vert   \ge \aa d(y_1, y_2).$$
By contradiction suppose that for every $1 \le m \le k-2$
\begin{equation}  \vert  \phi_m(y_1) - \phi_m(y_2) \vert    < \aa d(y_1, y_2). \label{NotDist}\end{equation}
We claim that in this case $\{y_1,y_2,x_2,\dots, x_{k-1}\}$ satisfies $\SRA(\aa)$-condition.
We have to check that (\ref{SRA-ineq}) is satisfied for every 
$x, y, z \in \{y_1,y_2,x_2,\dots, x_{k-1}\}$. We define a function 
$\wb \cdot :\{y_1,y_2,x_2,\dots, x_{k-1}\} \rightarrow \{x_1, x_2,\dots,x_{k-1}\}$ 
by 
$$\wb x_i = x_i\textit{, for $2 \le i \le k-1$,}$$
$$\wb y_1 = \wb y_2 = x_1.$$

\textit{Case A:} $x,y,z \subset  \{x_2,\dots, x_{k-1}\}$. This case is trivial since 
$\{x_1,\dots,x_{k-1}\} \subset X$ satisfies $\SRA(\frac{\aa}{2})$-condition.

\text{Case B:} $\{x,y,z\} =  \{y_1, x_i, x_j\}$. Since $\{x_1,\dots,x_{k-1}\} \subset X$ 
satisfies $\SRA(\frac{\aa}{2})$-condition. We have 
\begin{equation}
d(\wb x,\wb y) \le \max\{d(\wb x,\wb z) + \frac{\aa}{2} d(\wb z,\wb y), \frac{\aa}{2} d(\wb x,\wb z)
 + d(\wb z,\wb y)\}.
 \label{LocEmbEQ1}
 \end{equation}
Since  $y_1, y_2 \in B_{\frac{\aa R}{6}}(x_1)$ we have 
\begin{equation}
d(x, y) \le  d(\wb x,\wb y) + \frac{\aa R}{6}.
 \label{LocEmbEQ2}
 \end{equation}  
On the other hand since  $R = \min_{1 \le i < j \le k - 1}{d(x_i,x_j)}$ we have 

$$ \max\{d(\wb x,\wb z) + \frac{\aa}{2} d(\wb z,\wb y), \frac{\aa}{2} d(\wb x,\wb z) + d(\wb z,\wb y)\} \le $$
\begin{equation}
 \le  \max\{d(\wb x,\wb z) + \aa d(\wb z,\wb y), \aa d(\wb x,\wb z) + d(\wb z,\wb y)\} - \frac{\aa R}{2}.  
 \label{LocEmbEQ3}
 \end{equation}  
Once again we use that $y_1, y_2 \in B_{\frac{\aa R}{6}}(x_1)$
$$  \max\{d(\wb x,\wb z) + \aa d(\wb z,\wb y), \aa d(\wb x,\wb z) + d(\wb z,\wb y)\}  \le $$ 
\begin{equation}
\le \max\{d(x,z) + \aa d(z,y), \aa d(x,z) + d(z,y)\}  + 2\frac{\aa R}{6}.  
 \label{LocEmbEQ4}
 \end{equation}  
Adding together (\ref{LocEmbEQ1} - \ref{LocEmbEQ4}) provides
$$d(x, y) \le \max\{d(x,z) + \aa d(z,y), \aa d(x,z) + d(z,y)\}.$$

\text{Case C:} $\{x,y,z\} =  \{y_1, y_2, x_i\}$. If $z = x_i$ from $y_1, y_2 \in B_{\frac{\aa R}{6}}(x_1)$ 
we have that $d(x, y) \le 2\frac{\aa R}{6}$ and $d(x,z) \ge R - \frac{\aa R}{6}$. And (\ref{SRA-ineq})
follows.

 Finally the last case is $z = y_1$, $x = x_i$, $y = y_2$. We have to show that 
 $$d(x_i,y_2) \le \max\{d(x_i,y_1) + \aa d(y_1,y_2), \aa d(x_i,y_1) + d(y_1,y_2)\}. $$
 It suffices to show that 
 $$d(x_i,y_2) \le  d(x_i,y_1) + \aa d(y_1,y_2).$$
 Which follows directly from (\ref{NotDist}).
\end{proof}

\begin{proof}[Proof of Theorem \ref{EmbThm}]
First we need the following proposition.
\begin{proposition}[\cite{LOZ}, Theorem 6] \label{MeetDoubling}
For  $0 < \aa < 1$, $k \in \N$ there exist $L=L(k,\alpha)$ such that every 
metric space $(X,d)$, which is free of $k$-point $\SRA(\aa)$-subspaces, 
satisfies the doubling condition with the constant $L$.
\end{proposition}
We are going to prove the theorem via induction by $k$.
The base is that the theorem is true for $k = 3$.
\begin{proof}[Proof of  the base.] 
\textit{Case A:  $\diam(X) < \infty$.}  
Fix a maximal $\frac{\diam(X)}{10}$-separated set $\{x_1,\dots,x_n\}$ in $X$. 
By Proposition \ref{MeetDoubling} we have that $n$ is bounded from above by some $C(\aa)$.
We claim that a map $\phi:X \rightarrow \R^{n}$ given by 
$$\phi_i(y) = \dist(y, x_{i})$$
has distortion less or equal to $\sqrt{C(\aa)}\max\{10,\frac{1}{\aa}\}$.
Clearly $\Lip(\phi) \le \sqrt{n} \le \sqrt{C(\aa)}$. Thus it suffices to show that for every 
$y_1,y_2 \in X$ there exists $1 \le i' \le n$ such that 
\begin{equation}
\vert    \phi_{i'}(y_1) - \phi_{i'}(y_2)   \vert  \ge \min\{\frac{1}{10}, \aa\} d(y_1, y_2).
\label{BaseCoLip}
\end{equation}
Since $\{x_1,\dots,x_n\}$ is a maximal $\frac{\diam(X)}{10}$-separated set 
there exists $x_i$ such that $d(x_i, y_1) \le \frac{\diam(X)}{10}$.

If we have $d(y_1, y_2) \ge \frac{3\diam(X)}{10}$ then we can take $x_{i'} = x_i$.
Indeed, by the triangle inequality 
we have $d(x_i, y_2) \ge \frac{2\diam(X)}{10}$. Which implies that 
$$\vert  \phi_i(y_1) - \phi_i(y_2) \vert   \ge \frac{\diam(X)}{10}  \ge \frac{1}{10}d(y_1, y_2).$$
And (\ref{BaseCoLip}) follows.

If we have $d(y_1, y_2) \le \frac{3\diam(X)}{10}$ then there exists $x \in X$ such that 
$d(x_i,x) \ge \frac{\diam{X}}{2}$.  Once again we apply that $\{x_1,\dots,x_n\}$ is a 
maximal $\frac{\diam(X)}{10}$-separated set and obtain that there exists $x_j$ such 
that $d(x_j,x_i) \ge \frac{\diam{X}}{2} - \frac{\diam(X)}{10}$. In this case we can
take $x_{i'} = x_j$. We claim that 
\begin{equation}
\vert   \phi_j(y_1) - \phi_j(y_2) \vert    \ge \aa  d(y_1, y_2).
\label{BaseCoLip2}
\end{equation}
By contradiction suppose that 
\begin{equation}
\vert    \phi_j(y_1) - \phi_j(y_2) \vert  < \aa  d(y_1, y_2).
\label{BaseCoLip3}
\end{equation}
From (\ref{BaseCoLip3}) we have 
$$d(x_j, y_1)   <  d(x_j, y_2) + \aa  d(y_1, y_2),$$
$$d(x_j, y_2)   <  d(x_j, y_1) + \aa  d(y_1, y_2).$$
On the other hand we know that a subspace $\{x_i,y_1,y_2\}$ 
does not satisfy $\SRA(\aa)$-condition. Thus,
 $$d(y_1, y_2)  >  \max\{d(x_j, y_1) + \aa d(x_j, y_2),  d(x_j, y_2) + \aa d(x_j, y_1)\} \ge d(x_j, y_1).$$
By the triangle inequality we have 
$$d(x_j,  y_1) \ge d(x_j, x_i) - d(x_i, y_1) \ge \diam(X)(\frac{1}{2} - \frac{1}{10} - \frac{1}{10}) \ge 
\frac{3\diam(X)}{10}.$$
We conclude that $$d(y_1, y_2) > \frac{3\diam(X)}{10}.$$
Which contradicts with our assumption.

\textit{Case B: $\diam(X) = \infty$.} This case follows from the following standard limiting argument.
Fix a point $x \in X$. From Case A we have that there exist $N = N(\aa), D = D(\aa)$ 
such that for every $m \in \N$ there exist an embedding 
$\phi_n:B_{m}(x) \rightarrow \R^N$ which distortion does not exceed $D$ and such that 
$\phi(x) = 0$.

Since $X$ satisfies the doubling condition it is also separable. We fix a dense countable subset 
$Y \subset X$. By the diagonal argument there exists a map $\phi:Y \rightarrow \R^N$
having distortion less or equal to $D$. The continuous extension of $\phi$ provides the 
required embedding.
\end{proof}
From now we have to deal with the inductive step. We will show 
that theorem holds for $k \in \N$ and $0 < \aa < 1$ 
under assumption that it holds for $k - 1$ and $\frac{\aa}{2}$.
\begin{proof}
For a point $x \in X$ we define $R(x) \in [0, \infty]$ as the infimum of $R > 0$ such 
that every $Y \subset X$ satisfying the $\SRA(\frac{\aa}{2})$-condition and also such that
$x \in Y$, $\vert \vert Y \vert \vert = k - 1$ we have 
$$\min_{y_1 \neq y_2 \in Y}d(y_1, y_2) \le R.$$ 
We define $X' \subset X$ as a maximal subset of $X$ such that for every $x \neq y \in X'$,
$$d(x, y) > \min\{R(x), R(y)\}.$$ 

 \textit{Claim: $X'$ is free of $(k-1)$-point $\SRA(\frac{\aa}{2})$-subspaces.}
 \begin{proof}
This follows from definitions of $R$ and $X'$ in a tautological fashion.
 Indeed, suppose that $(k-1)$-point $Y \subset X'$ satisfies the $\SRA(\frac{\aa}{2})$-condition.
 Let $x \in Y$ be such that $$R(x) = \min_{y \in Y}R(y).$$
From the definition of $R(x)$ we conclude that there exist $y_1 \neq y_2 \in Y$ such that 
$$d(y_1, y_2) \le R(x).$$
Thus we have 
$$R(x) = \min_{y \in Y}R(y) \le \min\{R(y_1), R(y_2)\} < d(y_1, y_2) \le R(x).$$
Which is a contradiction.
\end{proof}

By the inductive assumption there exist $N' = N'(k-1,\frac{\aa}{2})$, $D' = D'(k-1,\frac{\aa}{2})$
 and a map $\phi:X' \rightarrow \R^{\N'}$ which bi-Lipschitz distortion does not exceed $D$.
We are going to construct the required embedding by extending $\phi$ via Theorem \ref{ExtThm}.

Fix a point $x \in X \setminus X'$. Note that that the only possible reason why we cannot 
add $x$ to $X'$ is that there exists $x' \in X'$ such that $d(x,x') < R(x).$
Thus we have 
\begin{equation} d(x, X') \le R(x). \label{Close} \end{equation}

From Lemma \ref{LocEmb} it follows that $B_{\frac{aR(x)}{6}}(x)$
allows an embedding into $\R^{k-2}$ with distortion $\frac{\sqrt{k-2}}{\aa }$.
Combined with (\ref{Close}) this provides all required ingredients  for Theorem \ref{ExtThm}.
So the induction step follows.
\end{proof}
\end{proof}

\section{Proof of Theorem \ref{AlexThm}}
{\color{Black}It was shown in \cite{LOZ} that balls in Alexandrov spaces are 
$\SRA(\aa)$-free for every $0 < \aa < 1$ via the $\ATB$-condition.}
To prove Theorem \ref{AlexThm} we are going to reformulate this result  
in a slightly stronger form. {\color{Black}To do this we take from \cite{LOZ}}
\begin{itemize}
\item{the definition of $\ATB(\ee)$-condition,}
\item{the theorem saying that metric spaces
 satisfying  $\ATB(\ee)$-condition are $\SRA(\aa)$-free,}
\item{and the theorem providing $\ATB(\ee)$-condition for Alexandrov spaces.}
\end{itemize}

\begin{definition}[Angular total boundedness; $\ATB(\ee)$]\label{ATB}
Let $(X,d)$ be a metric space and $0<\ee<\pi/2$.
We say that a point $p \in X$ satisfies the \emph{$\ATB(\ee)$-condition}
if there exist some $L \in \N$ and $R > 0$ such that,
for every $y_1,\dots,y_L \in B_R(p) \setminus \{p\}$,
we can find $i \neq j$ satisfying
\[ \wt\angle y_i p y_j < \ee, \]
where  $\wt \angle$ denotes the planar comparison angle.
We say that $X$ is a \emph{globally $\ATB(\ee)$-space}
if there exists $L \in \N$ such that every $p \in X$ satisfies the $\ATB(\ee)$-condition
with constants $L$ and $R = \infty$.
We say that $X$ is a \emph{locally $\ATB(\ee)$-space} if, for any $x \in X$,
there exist $L \in \N$ and $R > 0$ such that every $p \in B_R(x) $ satisfies the $\ATB(\ee)$-condition
with constants $L$ and $R$.
\end{definition}

\begin{proposition}[\cite{LOZ}, Theorem 2]\label{ATB-no-SRA} 
For every $L \in \N$ there exists $N(L) \in \N$ satisfying the following.
{\color{Black}For every $0<\aa<1$, $\ee :=\arccos(\aa)/2 > 0$, $r > 0$,
 every metric space $(X,d)$ and every $x \in X$ such that 
every $p \in B_r(x)$ satisfies the $\ATB(\ee)$-condition with constants $L$ and $2r$
we have that $B_r(x)$ is free of $N(L)$-point 
   $\SRA(\aa)$-subspaces.}   
\end{proposition}

The $\ATB(\ee)$-condition for Alexandrov spaces is provided by \cite[Proposition 14]{LOZ}.
Here we formulate this proposition in a stronger form then in \cite{LOZ}, but the 
same proof works.

\begin{proposition}[$\ATB$ of Alexandrov spaces]\label{CBBs-are-ATB}
For every $k < 0$, $0<\ee<\pi/2$  and $n \in \N$ there exist $L \in \N$ and $r > 0$ satisfying the following. 
For every $n$-dimensional Alexandrov space $X$ of curvature $\ge k$, 
every $x \in X$ satisfies the $\ATB(\ee)$-condition for $L$ and $r$.

For every $n \in \N$, $0<\ee<\pi/2$ there exists $L_0 \in N$ such that every point in
every   $n$-dimensional non-negatively curved Alexandrov space satisfies 
 $\ATB(\ee)$-condition for $L_0$ and $R_0 = \infty$. 
\end{proposition}

\begin{proof}[Proof of Theorem \ref{AlexThm}]
We only give a proof for the case $k < 0$, 
the proof for the case $k = 0$ uses the same arguments and it is even simpler.

We fix $n \in \N$, $k < 0$ and $R > 0$.  
Next we fix an arbitrary $0<\aa<1$ and set $\ee :=\arccos(\aa)/2 > 0$.
(For the proof of Theorem \ref{AlexThm} we can use the $\SRA$-free condition for any $0<\aa<1$). 

By Proposition \ref{CBBs-are-ATB} we have that there exist $L(k, \ee) \in \N$ and 
$r(k, \ee) > 0$ such that every point in $X$ satisfies the $\ATB(\ee)$-condition for $L$ and $r$.
By Proposition \label{ATB-no-SRA}  there exists $N(k, \ee)$ such that 
any open ball of radius $\frac{r}{2}$ in $X$ is free of $N$-point $\SRA(\aa)$-subspaces.
Since $X$ is an Alexandrov space there exists $\ll(n, k, R)$ such that any ball 
of radius $R$ satisfies doubling condition with a constant  $\ll$.
It is easy to see that two last statements imply that there exist $\wt N = \wt N(n, k, R, \ee) \in \N$
such that any ball of radius $R$ in $X$ is free of $\wt N$-point $\SRA(\aa)$-subspaces
(see \cite[Theorem 6]{ZSC} for details). Thus the existence of the required embedding 
follows from Theorem \ref{EmbThm}.
\end{proof}

\section{Open problems}
\begin{Question}[see \cite{RicciQuest}]
Let $n$ be a natural number. Is it true that there exist $N(n), D(n) > 0$
 such that any complete $n$-dimensional Riemannian manifold of nonnegative 
 Ricci curvature can be embedded into $N$-dimensional Euclidean space with 
 bi-Lipschitz distortion less then $D$?
\end{Question}
See also a similar question by  S. Eriksson-Bique \cite[Conjecture 1–11]{EB}.

\bibliography{circle}

\bibliographystyle{plain}

\end{document}